\documentclass[12pt]{article}
\usepackage{amsmath}
\usepackage{amssymb}
\usepackage{amsthm}
\usepackage{txfonts}
\usepackage{graphicx}
\usepackage{subfigure}
\usepackage{caption}
\usepackage{empheq}
\usepackage{algorithm}
\usepackage[noend]{algpseudocode}

\def\Ac{\mathcal{A}}
\def \E{\mathbb{E}}
\def \R{\mathbb{R}}
\def \M{\mathbb{M}}
\def \S{\mathbb{S}}
\def \W{\mathbb{W}}
\def \V{\mathbb{V}}
\def\Lc{{\cal L}}
\def\x{\times}
\def \Fb{\overline F}
\def\1{{\bf 1}}
\def \I{{\bf I}}
\def \N{\mathbb{N}}
\newcommand{\un}{1\hspace{-1mm}{\rm I}}   

\def\subclassname{{\bfseries Mathematics Subject Classification
(2000)}\enspace}
\def\subclass#1{\par\addvspace\medskipamount{\rightskip=0pt plus1cm
\def\and{\ifhmode\unskip\nobreak\fi\ $\cdot$
}\noindent\subclassname\ignorespaces#1\par}}

\def\keywords#1{\par\addvspace\medskipamount{\rightskip=0pt plus1cm
\def\and{\ifhmode\unskip\nobreak\fi\ $\cdot$
}\noindent\keywordname\enspace\ignorespaces#1\par}}
\def\keywordname{{\bfseries Keywords}}

\newtheorem{Theorem}{Theorem}[section]
\newtheorem{proposition}[Theorem]{Proposition}
\newtheorem{remark}[Theorem]{Remark}
\newtheorem{lemma}[Theorem]{Lemma}
\newtheorem{assump}{Assumption}

\title{Monte Carlo for high-dimensional degenerated Semi Linear and Full Non Linear PDEs}

\author{Xavier Warin}

\begin{document}

\maketitle

\abstract{We extend a recently  developed method to solve semi-linear PDEs to the case of a degenerated diffusion. Being a pure Monte Carlo method it does not suffer from the so called curse of dimensionality and it can be used to solve problems that were out of reach so far. We give some results of convergence and show numerically that it is effective. 
Besides we numerically show that the new scheme developed can be used to solve some full non linear PDEs. At last we provide an effective algorithm to implement the scheme.
\subclass{MSC 65C05 \and MSC  49L25 }
}

\section{Introduction}
The resolution of non linear PDEs in high dimension is challenging due to the curse of dimensionality. Deterministic methods cannot compete in dimension above 4 and the most  used  approach in moderate dimension is the BSDE approach first proposed in \cite{pardoux1990adapted}  that led to the time resolution scheme proposed in \cite{bouchard2004discrete} and to an effective global resolution scheme based on regression   in \cite{gobet2005regression} and \cite{lemor2006rate}.
The full-non linear case, always based on regression, was treated in \cite{fahim2011probabilistic}, \cite{tan2013splitting} following the representation proposed in \cite{cheridito2007second}.\\
All these methods cannot be used in dimension above 6 or 7: the regression is achieved by projecting some functions  on a space of basis functions with a cardinality exploding with the dimension of the problem. It is important to understand that the first problem encountered in high dimension is not the computational time used but the memory required by the algorithm: regression in dimension $d=7$ or $d=8$ requires to store  millions of particles in memory and by taking only $4$ basis functions in each direction, it leads to a global number of basis functions equal to $4^d$ so exploding very quickly.\\
Recently some new methods have been developed to solve non linear PDEs:
\begin{itemize}
\item Deep learning techniques have been recently proposed to solve semi-linear PDEs \cite{han2017overcoming}, \cite{weinan2017deep} and the methodology has been extended to full non linear equations in \cite{beck2017machine}. This approach appears to be effective but no result of convergence is available so its limitations are unknown.
\item In \cite{warin2018nesting}, a new scheme based on nesting Monte Carlo is proposed to solve semi-linear equations in high dimension. The ingredients of this method are the randomization of the time step proposed in \cite{henry2016branching},\cite{doumbia2017unbiased} and the automatic differentiation method used in \cite{henry2016branching} and that was first  proposed in \cite{fournie1999applications}.
In the scheme proposed in \cite{warin2018nesting} a truncation  is achieved after a given number of switches corresponding to a given depth of the nesting method. The scheme proposed is numerically effective. However, it cannot deal with degenerated diffusions.
\item In \cite{weinan2017linear}, \cite{weinan2016multilevel}, \cite{hutzenthaler2017multi}, the authors develop an algorithm based on Picard iterations, multi-level techniques and automatic differentiation to solve some high dimensional PDEs with non linearity in $u$ and $Du$. They give some convergence results and a lot of numerical examples show its efficiency in high dimension. However, to our knowledge, this methodology cannot be used with a degenerated diffusion.
\end{itemize}
In this article, extending the work  in \cite{warin2018nesting}, we propose a scheme to solve the semi-linear case when the diffusion is degenerated, and  study the error associated to this scheme. Besides, we provide an effective algorithm to
implement the scheme and the most effective scheme proposed in \cite{warin2018nesting} to deal with a non linearity in $Du$. Some numerical results
confirm the interest of the methodology.\\
At last the scheme proposed here can be used to solve some full non-linear PDEs.
The convergence of the scheme is not proved but some numerical examples 
show its efficiency.\\
In the article, we take the following notations:
$\M^d$ is the set of $d \times d$ matrices. $\S^d$ the set of symmetric elements of $\M^d$. $\1_d=(1,\cdots,1)^{\top} \in\R^d$, $I_d$ is the unit diagonal matrix of $\M^d$.
For $(A,b) \in \M^d \times \M^d$, we note $A:B =  trace(A B^{\top})$. For $A \in \M^d$, $||A||_2  = \sqrt{ \displaystyle{\sum_{i=1}^d \sum_{j=1}^d }A_{i,j}^2}$.\\
For $u= (u_{i_1,.., i_q})_{i_p=1,..,d, p=1,..,q}$ where each element $u_{i_1,.., i_q}$ is a $\R$ value function of $C(\R^d)$, 
\begin{align*}
|u|_\infty = \sup_{i_p=1,..,d, p=1,..,q} \sup_{x \in \R^d } |u_{i_1,.., i_q}(x)|.
\end{align*}
All numerical experiments are achieved on a cluster using 16 nodes with a total of 448 cores and MPI is used for parallelization.  The generation of random numbers  in parallel mode is  achieved using Tina's Random Number Generator Library \cite{bauke2011tina}.
 All computational times are given for a configuration of Intel Xeon CPU E5-2680 v4  2.40GHz (Broadwell).

\section{The general problem}
Our goal is to solve the general full non linear equation
 \begin{flalign}
 \label{eqPDEFull}
  (-\partial_tu-\Lc u)(t,x)  & = f(t,x,u(t,x),Du(t,x),D^2u(t,x)), \nonumber \\
  u(T,x)&=g(x), \quad  t<T,~x\in\R^d,
 \end{flalign}
 with
\begin{flalign*}
 \Lc u(t,x) := \mu Du(t,x) +  \frac{1}{2} \sigma \sigma^{\top} \!:\! D^2 u(t,x)
 \end{flalign*}
 so that $\Lc $ is the generator associated to 
 \begin{flalign*}
  X_t = x +  \mu t+ \sigma  dW_t,
 \end{flalign*}
 with $\mu \in \R^d$, and $\sigma \in \M^d$ is some constant matrix.\\
 In the whole article,  $\rho$ is the density of  a general random variable following a gamma law so  that  $\rho$  is bounded by below by a strictly positive value on any interval $[0,T]$:
 \begin{align}
 \rho(x)= \lambda^\alpha x^{\alpha-1} \frac{e^{-\lambda x}}{\Gamma(\alpha)},  1 \ge \alpha >0.
 \label{rho}
 \end{align}
 The associated  cumulated distribution function  is  $$F(x) =\frac{\gamma(\alpha,\lambda x)}{\Gamma(\alpha)}$$
 where $\gamma(s, x) =  \int_0^x t^{s-1} e^{-t} dt$ is the incomplete gamma function and $\Gamma(s)= \int_0^\infty t^{s-1} e^{-t} dt$ is the gamma function.\\
 The methodology  follows the ideas of \cite{warin2018nesting} and \cite{warin2017variations}.
 The case where $f$ only depends on $u$ and $Du$ and $\sigma$ is invertible has been treated in \cite{warin2018nesting} and it has been shown that using a Gamma law is $\alpha<1$ the  method was converging. Besides numerically it was shown that the use of an exponential law corresponding to the limit case $\alpha=1$ was optimal.\\
 
\section{The general scheme}
In this section we first present the general scheme used to solve the problem.
We then give the general algorithm used. We suppose here that $\sigma$ is non degenerated so that $\sigma^{-1}$ exists.\\
Let set $p \in \N^{+}$.
For $(N_0,.., N_{p-1}) \in \N^p$, we introduce  the sets of i-tuple,
$Q_i = \{ k=(k_1, ...,k_i)\}$ for $i \in \{1,..,p\}$ where all components $k_j \in [1, N_{j-1}]$.
Besides we define $Q^p= \cup_{i=1}^p Q_i$.\\
We construct the sets $Q^o_i$  for $i = 1, .., p$, such that
$$Q^o_1= Q_1$$  and 
the set $Q^o_i$  for $i>1$ are defined by recurrence : 
\begin{flalign*}
Q^o_{i+1} = \{ (k_1,..,k_i, k_{i+1}) / (k_1,..,k_i) \in Q^o_{i}, k_{i+1} \in \{ 1,..,N_{i+1}, 1_{1},.., (N_{i+1})_1, 1_2,...,(N_{i+1})_2 \} \}
\end{flalign*}
 so that to a particle noted $(k_1,.., k_i)  \in Q_i^o$ such that $k_i \in \N$, we associate two fictitious particles noted $ k^1 = (k_1,..,k_{i-1}, (k_i)_{1})$  and 
 $ k^2 = (k_1,..,k_{i-1}, (k_i)_{2})$. \\
To a particle $k=(k_1, .., k_i) \in Q^o_i$ we associate its original particle $o(k) \in Q_i$ such that $o(k)= (\hat k_1,.. \hat k_i)$ where $\hat k_j = l$ if $k_j = l$, $l_{1}$ or $l_2$.\\
For $k=(k_1, ...,k_i) \in Q^o_i$ we introduce the set of its non fictitious sons
\begin{align*}
\tilde Q(k) = \{ l =(k_1,..,k_i, m )/  m \in \{1,.., N_i\} \} \subset Q_{i+1}^o,
\end{align*}
and the set of all sons
\begin{align*}
\hat Q(k) = \{ l =(k_1,..,k_i, m )/  m \in \{1,.., N_i, 1_1,..., (N_i)_1,1_2,...,(N_i)_2 \} \} \subset Q_{i+1}^o.
\end{align*}
By convention $\tilde Q(\emptyset) = \{ l =( m )/  m \in \{1,.., N_0\} \} =  Q_{1}.$
 Reciprocally the ancestor $k$ of a particle  $\tilde k$ in $\tilde Q(k)$ is noted $\tilde k^{-}$.\\
We define the order  of  a particle $k  \in Q^o_{i}$, $i \ge 0$,  by the function $\kappa$:
\begin{align*}
\kappa(k) =& 0 \mbox{ for } k_i \in \N, \\
\kappa(k) =  &  1  \mbox{ for } k_i = l_1,  l \in \N \\
\kappa(k) =  &  2 \mbox{ for } k_i = l_2 , l \in \N 
\end{align*}
We define the  sequence $\tau_{k}$  of switching increments that are i.i.d. random variables with density $\rho$ 
 for $k\in Q^p$.
 The switching dates are defined as :
 \begin{equation}
 \left\{
 \begin{array}{lll}
 T_{(j)} &  =&  \tau_{(j)} \wedge T,  j \in \{ 1,., N_{0}\} \\
 T_{ \tilde k} & =&  ( T_{k} + \tau_{ \tilde k}) \wedge T,  k =( k_1,..k_i) \in Q_i, \tilde k \in \tilde Q(k)
 \end{array}
 \right.
 \end{equation}
By convention $T_k = T_{o(k)}$ and $ \tau_k = \tau_{o(k)}$.
For $k = (k_1,..,k_i) \in Q^o_i$  and $\tilde k =(k_1,..,k_i, k_{i+1})  \in \hat Q(k)$ we define the following trajectories :
	\begin{flalign}\label{eq:brownRenorm}
		W^{\tilde k}_s
		~:=~&
		W^{k}_{T_{k}}
		~+~
		\1_{ \kappa( \tilde k) =0}
		\bar W^{o(\tilde k)}_{s - T_{k}} 
        ~-~
		\1_{ \kappa(\tilde k)=1}
		\bar W^{o(\tilde k)}_{s - T_{k}},
		~~~\mbox{and}~~\\
		X^{\tilde k}_s := &x +\mu s +\sigma  W^{\tilde k}_s,
		~~~\forall s \in [T_{k}, T_{\tilde k}],
	\end{flalign}
where the $\bar W^{ k}$ for $k$ in $Q^p$  are independent $d$-dimensional Brownian motions, independent of the $(\tau_{k})_{k \in Q^p}$.\\
In order to understand what these different trajectories represent, suppose that  $d=1$, $\mu= 0$, $\sigma =1$ and let us consider the original particle $k= (1,1,1)$  such that $T_{(1,1,1)}=T$.\\
Following equation \eqref{eq:brownRenorm},
\begin{align*}
X^{(1,1,1)}_T = & \bar W^{(1)}_{T_{(1)}} + \bar W^{(1,1)}_{T_{(1,1)}-T_{(1)}} + \bar W^{(1,1,1)}_{T-T_{(1,1)}} \\
X^{(1_{1},1,1)}_T= & -\bar W^{(1)}_{T_{(1)}} + \bar W^{(1,1)}_{T_{(1,1)}-T_{(1)}} + \bar W^{(1,1,1)}_{T-T_{(1,1)}} \\
X^{(1,1_{1},1)} = & \bar W^{(1)}_{T_{(1)}} - \bar W^{(1,1)}_{T_{(1,1)}-T_{(1)}} + \bar W^{(1,1,1)}_{T-T_{(1,1)}} \\
X^{(1_{2},1_{1},1)}_T = & - \bar W^{(1,1)}_{T_{(1,1)}-T_{(1)}} + \bar W^{(1,1,1)}_{T-T_{(1,1)}} \\
...&
\end{align*}
such that all particles are generated from the $\bar W^k$ used to define $X^{(1,1,1)}_T$.\\
Using the previous definitions,
we consider the estimator defined by:
\begin{empheq}[left=\empheqlbrace] {align}
\bar u_\emptyset^p = & \frac{1}{N_0} \sum_{j=1}^{N_0}   \phi\big( 0, T_{(j)}, X^{(j)}_{T_{(j)}}, \bar u_{(j)}^p, D \bar u_{(j)}^p, D^2 \bar u_{(j)}^p\big) , \nonumber\\
\bar u_{k}^p = &  \frac{1}{N_i} \sum_{\tilde k \in \tilde Q(k)}   \frac{1}{2} \big(  \phi\big(T_k,T_{\tilde k},X^{\tilde k}_{T_{\tilde k}},\bar u_{\tilde k}^p, D\bar u_{\tilde k}^p, D^2\bar u_{\tilde k}^p\big) + \nonumber \\
& \quad \quad \phi\big(T_k,T_{\tilde k},  X^{\tilde k^{1}}_{T_{\tilde k}}, \bar u_{\tilde k^{1}}^p D\bar u_{\tilde k^{1}}^p, D^2\bar u_{\tilde k^{1}}^p\big)  \big), \quad \mbox{ for }  k =( k_1,...,k_i) \in Q^o_i, 0 < i <p, \nonumber \\
D \bar u_{k}^p = &  \frac{1}{N_i} \sum_{\tilde k \in \tilde Q(k)}   \V^{\tilde k} \frac{1}{2} \big(  \phi\big(T_k,T_{\tilde k},X^{\tilde k}_{T_{\tilde k}},\bar u_{\tilde k}^p, D\bar u_{\tilde k}^p, D^2\bar u_{\tilde k}^p\big) -\nonumber \\
& \quad \quad  \phi\big(T_k,T_{\tilde k},  X^{\tilde k^{1}}_{T_{\tilde k}}, \bar u_{\tilde k^{1}}^p D\bar u_{\tilde k^{1}}^p, D^2\bar u_{\tilde k^{1}}^p\big)  \big), \quad \mbox{ for }  k =( k_1,...,k_i) \in Q^o_i, 0 < i < p, \nonumber \\
D^2\bar u_{k}^p = &  \frac{1}{N_i} \sum_{\tilde k \in \tilde Q(k)}   \W^{\tilde k} \frac{1}{2} \big(  \phi\big(T_k,T_{\tilde k },X^{\tilde k}_{T_{\tilde k}}, \bar u_{\tilde k}^p, D\bar u_{\tilde k}^p, D^2\bar u_{\tilde k}^p\big) + \nonumber \\
& \quad \quad \phi\big(T_k,T_{\tilde k},  X^{\tilde k^{1}}_{T_{\tilde k}}, \bar u_{\tilde k^{1}}^p, D\bar u_{\tilde k^{1}}^p, D^2\bar u_{\tilde k^{1}}^p\big) - \nonumber \\
& \quad \quad 2 \phi\big(T_k,T_{\tilde k},  X^{\tilde k^{2}}_{T_{\tilde k}}, \bar u_{\tilde k^{2}}^p, D\bar u_{\tilde k^{2}}^p, D^2\bar u_{\tilde k^{2}}^p\big) \big), \quad \mbox{ for }  k =( k_1,...,k_i) \in Q^o_i, 0 < i <p, \nonumber \\
\bar u_{k}^p = &  g(X^{k}_{T_{k}}), \quad  \mbox{for } k \in Q_p^o, \nonumber \\
D\bar u_{k}^p = & D g(X^{ k}_{T_{k}}), \quad  \mbox{for } k \in Q_p^o, \nonumber \\
D^2\bar u_{k}^p = & D^2 g(X^{k}_{T_{ k}}), \quad  \mbox{for }  k \in Q_p^o,
\label{eq:estimFull1}
\end{empheq} 
 where $\phi $ is defined by  :
 \begin{flalign}
 \label{defPhi}
  \phi(s, t,x,y,z,\theta) &:= \frac{\1_{\{t\ge T\}}}{ \Fb(T-s)} g(x)\!+\! \frac{\1_{\{t<T\}}}{\rho(t -s)} f(t,x,y,z,\theta),
 \end{flalign}
 and 
 \begin{align*}
 \V^{k} = \sigma^{-\top} \frac{\bar W_{ T_{k}- T_{k^{-}}}^{k}}{T_{ k}- T_{ k^{-}}},
 \end{align*},
 \begin{align}
 \W^{k} = 
  (\sigma^{\top})^{-1} \frac{\bar W^{k}_{T_{k}- T_{k^{-}}}(\bar W^{ k}_{T_{ k}- T_{k^{-}}})^{\top} - (T_{ k}- T_{k^{-}})  I_d}{(T_{ k}- T_{k^{-}})^2} \sigma^{-1}.
 \end{align}
 As explained before, the $u$ and $Du$ term in $f$ are treated as explained in \cite{warin2018nesting} and only the $D^2u$ treatment is the novelty of this scheme.
 \begin{remark}
 In practice, we just have the $g$ value at the terminal date $T$ and we want to apply the scheme even if the derivatives of the final solution is not given.
 We can close the system for $k$ in $Q_p^o$  replacing $\phi$ by $g$ and taking some value for $N_{p+1}$:
 \begin{align*}
 \bar u_{k}^p = &  \frac{1}{N_{p}} \sum_{\tilde k \in \tilde Q(k)}   \frac{1}{2} \big(  g\big(X^{\tilde k}_{T_{\tilde k}}\big) + g \big( X^{\tilde k^{1}}_{T_{\tilde k}}\big)  \big) , \\
D\bar u_{ k}^p = & \frac{1}{N_{p}} \sum_{\tilde k \in \tilde Q(k)}  \V^{\tilde k} \frac{1}{2} \big(  g\big(X^{\tilde k}_{T_{\tilde k}}\big) -g\big(  X^{\tilde k^{1}}_{T_{\tilde k}}\big) \big), \\
D^2\bar u_{k}^p = &  \frac{1}{N_{p}} \sum_{\tilde k \in \tilde Q(k)} 
 \W^{\tilde k}  \frac{1}{2} \big(  g \big(X^{\tilde k}_{T_{\tilde k}}\big) +
g \big( X^{\tilde k^{1}}_{T_{\tilde k}}\big) -  2 g\big( X^{\tilde k^{2}}_{T_{\tilde k}}\big) \big).
 \end{align*}
 In all our numerical examples, we use this approximation.
 \end{remark}
 \begin{remark}
 In the case where the coefficient are not constant, some Euler scheme can be added as explained in \cite{warin2018nesting}.
 \end{remark}
 An effective algorithm for this scheme  is given these two functions:
 \begin{algorithm}[H]
\caption{\label{algoMC1} Outer Monte Carlo algorithm ($V$ generates unit Gaussian RV, $\tilde V$ generates RV with gamma law density)}
\begin{algorithmic}[1]
\Procedure{PDEEval}{$\mu$, $\sigma$, $g$, $f$, $T$, $p$, $x_0$, $\{N_0,.., N_{p+1}\}$, $V$, $\tilde V$} 
\State $u_M = 0$ 
\State $x(0,:)= x_0(:)$ \Comment{$x$ is a matrix of size $1 \times n$}
\For{$i = 1, N_0$}
\State   $(u,Du,D^2u)=$ EvalUDUD2U$(x_0,\mu, \sigma,g, T,\{N_0,.., N_{p+1}\}, V,\tilde V,p,1, 0,0)$ 
\State $u_M = u_M + u(0)$
\EndFor
\Return $\frac{U_M}{N_0}$
\EndProcedure
\end{algorithmic}
\end{algorithm}
\begin{algorithm}[H]
\caption{\label{algoMC2} Inner Monte Carlo algorithm where $t$ is the current time, $x$ the array of particles positions of size $m \times d$, and $l$ the nesting level.}
\begin{algorithmic}[1]
\Procedure{EvalUDUD2U}{$x,\mu, \sigma,g, T,\{N_0,.., N_{p+1}\},V,\tilde V,p,m,t,l$} 
\State $\tau = \min ( \tilde V(),T-t)$,  \Comment{Sample the time step}
\State $G=V()$ \Comment{Sample the $n$ dimensional Gaussian vector}
\State  $xS(1:m,:) =  x(:)+ \mu \tau + \sigma G \sqrt{\tau}$
\State  $xS(m+1:2m,:)=   x(:)+ \mu \tau $
\State  $xS(2m+1:3m,:)=   x(:)+ \mu \tau -\sigma G \sqrt{\tau}$
\State $tS =t + \tau$ \Comment{New date}
\If {$ ts \ge T $ or $l = p$}
\State $g_1 = g(xS(1:m,:)) ; g_2 = (xS(m+1:2m,:)) ; g_3 = g(xS(2m+1:3m,:)) $ 
\State $ u(:)= \frac{1}{2} (g_1+g_3)$
\State $ Du(:,:)=  \frac{1}{2} (g_1-g_3) \quad \sigma^{- \top} G$
\State $D^2u(:,:,:)=  \frac{1}{2} (g_1+g_3-2g_2) \sigma^{- \top} \frac{GG^{\top}- \I_d}{\tau} \sigma^{-1}$
\If { $l \ne p$}
\State $(u(:),Du(:,:), D^2u(:,:,:)) /= \frac{1}{\bar F(\tau)}$ 
\EndIf
\Else
\State $ y(:)= 0; z(:,:) = 0; \theta(:,:,:) =0$
\For {$j=1, N_{l+1}$}
\State $(y,z, \theta)+=$EvalUDUD2U$(xS,\mu,\sigma,g,T,\{N_0,.., N_{p+1}\},V,\tilde V,p,3m,tS,l+1)$
\EndFor
\State $(y,z,\theta) /= N_{l+1}$
\For {$q= 1, m$}
\State $f_1 = f(ts,xS(q),y(q),z(q,:),\theta(q,:,:))$
\State $f_2 = f(ts,xS(m+q),y(m+q),z(m+q,:),\theta(m+q,:,:))$
\State $f_3 = f(ts,xS(2m+q),y(2m+q),z(2m+q,:),\theta(2m+q,:,:))$
\State $u(i)= \frac{1}{2}(f_1+f_3)$
\State $Du(i,:)= \frac{1}{2}(f_1-f_3) \sigma^{- \top} G $
\State  $D^2u(i,:,:)= \frac{1}{2}(f_1+f_3-2f_2) \sigma^{- \top} \frac{GG^{\top}- \I_d}{\tau} \sigma^{-1} $
\EndFor
\EndIf
\Return $(u,Du,D^2u)$
\EndProcedure
\end{algorithmic}
\end{algorithm}

 \section{The linear case}
 In this section  we suppose that $f$ is linear such that 
 \begin{align}
 f(\gamma)= A:\gamma , \mbox{ for } \gamma \in \S^d, A \in \M^d.
 \label{flin}
 \end{align}
 For an index $k =(k_1,.., k_i) \in Q_i^0$ we introduce 
 \begin{align}
 \label{eq::Diesek}
 \#k = \sum_{j=1}^i 1_{k_j = l_2, l \in \N},
 \end{align}
 and for $k \in Q_i$, the set of particles generated from an original particle $k$ by:
 \begin{align*}
 R(k) = \{  \bar k \in Q^o_i, o(\bar k) = k\}
 \end{align*}
 We make the following assumptions:
 \begin{assump}
 \label{ass::uProp}
 Equation  \eqref{eqPDEFull} has a solution $u$ such that
 \begin{itemize}
 \item $u \in C^{1,2 p}([0,T] \x \R^d)$ with uniformly bounded derivatives in $x$ and $t$.
 \item $D^{2i}u$ is $\theta$-H\"older with $\theta \in (0,1]$ in time with constant $\hat K$ for $i=1$ to $p$:
 \begin{align}
 | D^{2i}u(t,.) - D^{2i} u(\tilde t, .) |_\infty \le \hat K | t- \tilde{t}|^\theta \quad \quad \quad \forall (t,\tilde{t}) \in [0,T] \times [0,T].
 \label{eq:hol}
 \end{align}
 \end{itemize}
 \end{assump}
 
For $k= (k_1,..,k_i) \in Q_i$, $i \ge 1$, $u \in C^{2p}([0,T] \x \R^d)$ we introduce 
\begin{align}
\hat u^k = \frac{1}{2^{i-1}} \displaystyle{\sum_{ \tilde k \in R(k)}} u(T_k, X_{T_k}^{\tilde k}) (-2)^{\# \tilde k},
\end{align}
so that  for example :
\begin{itemize}
\item for  $k =(l) \in Q_1$,
$\hat u^k = u(T_{(l)}, X^{(l)}_{T_{(l)}})$,
\item for $k \in Q_2$,
\begin{align*}
\hat u^k =&  \frac{1}{2}(u(T_{k}, X^{k}_{T_{k}}) +
u(T_{k}, X^{k^1}_{T_{k}}) -2  u( T_{k}, X^{k^2}_{T_{k}})),
\end{align*}
\item
or $k =(l_1,l_2, l_3) \in Q_3$,
\begin{align*}
\widehat{u}^k =&  \frac{1}{4} \big( u(T_{k}, X^{(l_1,l_2,l_3)}_{T_{k}}) +
u(T_{k}, X^{(l_1,l_2,(l_3)_1)}_{T_{k}}) - 2 u(T_{k}, X^{(l_1,l_2,(l_3)_2)}_{T_{k}}) + \\
& u(T_{k}, X^{(l_1,(l_2)_1,l_3)}_{T_{k}}) +
u(T_{k}, X^{(l_1,(l_2)_1,(l_3)_1)}_{T_{k}}) - 2 u(T_{k}, X^{(l_1,(l_2)_1,(l_3)_2)}_{T_{k}}) -\\
& 2 u(T_{k}, X^{(l_1,(l_2)_2,l_3)}_{T_{k}}) - 2
u(T_{k}, X^{(l_1,(l_2)_2,(l_3)_1)}_{T_{k}}) + 4 u(T_{k}, X^{(l_1,(l_2)_2,(l_3)_2)}_{T_{k}}) \big).
\end{align*}
\end{itemize}
At last for $k = (k_1,..,k_i) \in Q_i$, $i > 1$ we introduce the 
set of all ancestors of $k$ plus $k$ and except the particle at the first level :
\begin{align*}
\hat An(k)= \{  (k_1,k_2) , ..., (k_1,..,k_{i}) \}
\end{align*}
We need a lemma to prepare the  result.
\begin{lemma}
\label{lemma1}
Suppose that $u \in C^{1,2 p-2}([0,T] \x \R^d)$, with uniformly bounded derivatives in $x$ and $t$ then there exists a positive constant $C(\sigma)$  such that for all $k \in Q_i$, $ 2 \le i \le p$, any interval $I$ of $\R$
\begin{align*}
\E \big( 1_{T_{ k} \in I} \quad  (\widehat{u}^k)^2  \displaystyle{\prod_{ \tilde k  \in \hat An(k)}} ||\W^{\tilde k}||^2_2 \big) \le  C(\sigma)^{i-1}  \sup_{t\in[0,T]}| D^{2(i-1)}u(t,.)|_\infty^2\E \big( 1_{T_{\tilde k} \in I} \big).
\end{align*}
\end{lemma}
\begin{proof}
For  $k \in Q_2$, using the mean value theorem
\begin{align*}
\E[  1_{T_{ k} \in I}  ||\W^{ k}||^2_2  (\widehat{u}^k)^2 ] =& 
\E[  1_{T_{ k} \in I}   ||\W^{ k}||^2_2  \frac{1}{4} (u(T_{k}, x + \mu T_{k} + \\ 
& \sigma  W^{k}_{T_k} ) +
u(T_{k}, x + \mu T_{k} +
 \sigma  W^{k^1}_{T_{k}} )  -2 u(T_{k}, x + \mu T_{k} +
 \sigma  W^{k^2}_{T_{k}} ))^2 ] \\
& \le d \: \E[ 1_{T_{\tilde k} \in I}   || \W ^k||^2_2  ||\sigma \bar W^{ k}_{\tau_{k}}||^4_2]  \sup_{t\in[0,T]} |  D^{2}u(t,.)|_\infty^2
\end{align*}
Then notice that $ || \W ^k||^2_2  ||\sigma \bar W^{k}_{\tau_{k}}||_2^4$ is independent of $\tau_{k}$ and $T_k$ such that 
\begin{align*}
E[ 1_{T_{\tilde k} \in I}   || \W ^k||^2_2  ||\sigma \bar W^{ k}_{\tau_{k}}||^4_2] ] =\E[ 1_{T_{\tilde k} \in I}] \E[  || \W ^k||^2_2  ||\sigma \bar W^{ k}_{\tau_{k}}||^4_2 ]
\end{align*}
and taking
$ C(\sigma)= d \: E[ || \W ^k||^2_2  ||\sigma \bar W^{k}_{\tau_{k}}||^4_2]$  we get the result.\\
Similarly using some multidimensional Taylor expansions, the independence of the $\bar W^l$  we get the result for $k \in Q_i$, $i>1$.
\qed
\end{proof}

We give the converging result in the linear case
\begin{proposition}
\label{theo1}
Under assumption \ref{ass::uProp}, supposing \eqref{flin} holds, there exists some functions of $u$: $C_1(u)$, $C_2(u)$, $C_3(u)$, and two functions $\hat C(T)$ and $C(\sigma)$ such that we have the following error given by the estimator \eqref{eq:estimFull1}:
\begin{align}
\E \big((\bar u_{\emptyset}^p - u(0,x))^2 \big) \le  
C_1(u) \hat C(T)^{2p} C(\sigma)^{p-1}  ||A||^{2p}_2 T^{2 \theta} \frac{\gamma(\alpha, \lambda T p)}{\Gamma(\alpha)} + \nonumber \\
\sum_{i=0}^{p-1} \frac{C_2(u)}{N_{i}} \hat C(T)^{2i+2}  C(\sigma)^{i} ||A||^{2i+2}_2\frac{\gamma(\alpha, \lambda T (i+1))}{\Gamma(\alpha)} + \nonumber \\
\sum_{i=0}^{p-1} \frac{C_3(u)}{N_i} \frac{\hat C(T)^{2i}  ||A||^{2i}_2 C(\sigma)^i}{\bar F(T)^2} 
\frac{\gamma(\alpha, \lambda T i)}{\Gamma(\alpha)}
\end{align}
\end{proposition}
\begin{proof}
The demonstration is in spirit similar to demonstration of propositions 2.3, 3.5 and 3.9
in \cite{warin2018nesting}. We only sketch the proof only highlighting the differences.\\
First notice that due  to assumption \ref{ass::uProp}, the solution $u$ of \eqref{eqPDEFull} satisfies  a Feynman-Kac relation  (see an adaptation of proposition 1.7 in \cite{touzi2012optimal} ) so that
for all $k \in Q_i$,  and $\forall \tilde k \in \tilde{Q}(k)$,
\begin{align*}
u(T_k, X^k_{T_k}) =& \E_{ T_{k}, X^k_{T_k}} \big[\phi\big(T_k , T_{\tilde k}, X_{T_{\tilde k}}^{\tilde{k}}, D^2u(T_{\tilde k},X_{T_{\tilde k}}^{\tilde{k}} ))\big],
\end{align*}
where 
\begin{flalign*}
  \phi(s, t,x,\theta) &:= \frac{\1_{\{t\ge T\}}}{ \Fb(T-s)} g(x)\!+\! \frac{\1_{\{t<T\}}}{\rho(t -s)} A:\theta.
 \end{flalign*}
Similarly  using automatic differentiation,
\begin{align*}
D^2u(T_k, X^k_{T_k}) =& \E_{ T_{k}, X^k_{T_k}} \big[ \W^{\tilde k} \phi\big(T_k , T_{\tilde k}, X_{T_{\tilde k}}^{\tilde{k}}, D^2u(T_{\tilde k},X_{T_{\tilde k}}^{\tilde{k}} ) \big)  \big],
\end{align*}
or 
\begin{align}
D^2u(T_k, X^k_{T_k}) =& \E_{ T_{k}, X^k_{T_k}} \big[ \W^{\tilde k} (\phi\big(T_k , T_{\tilde k}, X_{T_{\tilde k}}^{\tilde{k}}, D^2u(T_{\tilde k},X_{T_{\tilde k}}^{\tilde{k}} )\big)  - \phi\big(T_k , T_{\tilde k}, X_{T_{\tilde k}}^{\tilde{k}^2}, D^2u(T_{\tilde k},X_{T_{\tilde k}}^{\tilde{k}^2} ))\big) \big],
\label{eq:D2U}
\end{align}
where $\phi\big(T_k , T_{\tilde k}, X_{T_{\tilde k}}^{\tilde{k}^2}, D^2u(T_{\tilde k},X_{T_{\tilde k}}^{\tilde{k}^2} ))$ acts as a control variate.\\
Using the antithetic random variables:
\begin{align}
D^2u(T_k, X^k_{T_k}) =& \E_{ T_{k}, X^k_{T_k}} \big[ \W^{\tilde k} (\phi\big(T_k , T_{\tilde k}, X_{T_{\tilde k^1}}^{\tilde{k}}, D^2u(T_{\tilde k},X_{T_{\tilde k}}^{\tilde{k}^1} )\big)  - \phi\big(T_k , T_{\tilde k}, X_{T_{\tilde k}}^{\tilde{k}^2}, D^2u(T_{\tilde k},X_{T_{\tilde k}}^{\tilde{k}^2} ))\big) \big],
\label{eq:D2UA}
\end{align}
so that another representation is obtained  by adding
\eqref{eq:D2U} and \eqref{eq:D2UA}: 
\begin{align*}
D^2u(T_k, X^k_{T_k}) =& \frac{1}{2}\E_{ T_{k}, X^k_{T_k}} \big[ \W^{\tilde k} (\phi\big(T_k , T_{\tilde k}, X_{T_{\tilde k}}^{\tilde{k}}, u(T_{\tilde k},X_{T_{\tilde k}}^{\tilde{k}} )) + \phi\big(T_k , T_{\tilde k}, X_{T_{\tilde k}}^{\tilde{k}^1}, u(T_{\tilde k},X_{T_{\tilde k}}^{\tilde{k}^1} )) - \\
& 2  \phi\big(T_k , T_{\tilde k}, X_{T_{\tilde k}}^{\tilde{k}^2}, u(T_{\tilde k},X_{T_{\tilde k}}^{\tilde{k}^2} ))) \big].
\end{align*}
Introduce  for $k \in  Q_i$, $0<i<p$:
\begin{flalign*}
E_k := & \E_{T_{k},X^{k}_{T_{k}} }\big(  || \frac{1}{2^{i-1}}\sum_{ \bar k \in R(k) } (D^2\bar u_{ \bar k}^p - D^2u(T_k,X^{\bar k}_{T_k})) (-2)^{\#\bar k} ||^2_2 1_{T_k <T} \big) \nonumber\\
\end{flalign*}
with the convention $E_\emptyset = \E[ (\bar u_{\emptyset}^p -u(0,x))^2  ]$.\\
Using the methodology used in \cite{warin2018nesting} (see proposition equation  (2.26) in this article):
\begin{align}
 \E\big((\bar u_{\emptyset}^p - u(0,x))^2 \big) \le& \frac{1}{N_0} (1+\frac{8}{N_0}) \sum_{\tilde k \in \tilde Q(\emptyset)} \E \big( \frac{1_{T_{\tilde k < T}}}{\rho(\tau_{\tilde k})^2}  (A : (D^2 \bar u_{\tilde k}^p - D^2u(T^{\tilde k},X^{\tilde k}_{T_{\tilde k}})))^2 \big) + \nonumber
\\
& 4 \frac{1}{N_0^2} \sum_{\tilde k \in \tilde Q(\emptyset)} \E \big(  1_{T_{\tilde k < T}} (\frac{A:D^2u(T_{\tilde k},X^{\tilde k}_{T_{\tilde k}})}{\rho(\tau_{\tilde k})} )^2\big) +  \nonumber \\
&2 \frac{1}{N_0^2} \sum_{\tilde k \in \tilde Q(\emptyset)} \E_{T_k,X^k_{T_k} }\big( 1_{T_{\tilde k\ge T}} \frac{g(X^{\tilde k}_T)^2}{ \Fb(T-T_k)^2} \big).
\label{eq:E0}
\end{align}
so that using  discrete Cauchy Schwartz and noting that
$ E_{\tilde k} =E_{T_{ \tilde k},X^{\tilde k}_{T_{\tilde k}} } \big(  || D^2\bar u_{ \tilde k}^p - D^2u(T_{\tilde k},X^{\tilde k}_{T_{\tilde k}})||^2_2 1_{T_{\tilde k} <T} \big)$
\begin{align*}
 \E\big((\bar u_{\emptyset}^p - u(0,x))^2 \big) \le& \frac{1}{N_0} (1+\frac{8}{N_0}) \sum_{\tilde k \in \tilde Q(\emptyset)} \E \big( \frac{||A||^2_2}{\rho(\tau_{\tilde k})^2} E_{\tilde k} \big) + \nonumber
\\
& 4 \frac{1}{N_0^2} \sum_{\tilde k \in \tilde Q(\emptyset)} \E \big(  1_{T_{\tilde k < T}} \frac{||A||_2^2}{\rho(\tau_{\tilde k})^2} ||D^2u(T_{\tilde k},X^{\tilde k}_{T_{\tilde k}}))||^2_2\big) +  \nonumber \\
&2 \frac{1}{N_0^2} \sum_{\tilde k \in \tilde Q(\emptyset)} \E_{T_k,X^k_{T_k} }\big( 1_{T_{\tilde k\ge T}} \frac{g(X^{\tilde k}_T)^2}{ \Fb(T-T_k)^2} \big).
\end{align*}
In the same manner, for $k \in Q_i, i >0$, and using that $f$ is a linear operator:
\begin{align*}
E_k \le& \frac{1}{N_i} (1+\frac{8}{N_i}) \sum_{\tilde k \in \tilde Q(k)} E_{T_k,X^{k}_{T_{k}}} \big(  1_{T_{\tilde k < T}}  \frac{||\W^{\tilde k}||^2_2} {\rho(\tau_{\tilde k})^2} (A: \frac{1}{2^{i}} \big( \sum_{ \bar k \in R(\tilde k) } (D^2 \bar u_{\bar k}^p - D^2u(T^{\bar   k},X^{\bar k}_{T_{\tilde k}}))(-2)^{\#\bar k}) )^2 \big) + \nonumber \\
& 4 \frac{1}{N_i^2} \sum_{\tilde k \in \tilde Q(k)} \E_{T_k,X^k_{T_k} }\big(  1_{T_{\tilde k < T}} ||\W^{\tilde k}||^2_2   \frac{( A : \widehat{D^2u}^{\tilde k} )^2\big)}{\rho(\tau_{\tilde k})^2} +  \nonumber \\
&2 \frac{1}{N_i^2} \sum_{\tilde k \in \tilde Q(k)} \E_{T_k,X^k_{T_k} }\big( 1_{T_{\tilde k\ge T}}||\W^{\tilde k}||^2_2 \frac{ (\hat{g}^{\tilde k})^2}{ \Fb(T-T_k)^2} \big).
\end{align*}

We deduce using discrete Cauchy Schwartz that
\begin{align*}
E_k \le& \frac{1}{N_i} (1+\frac{8}{N_i}) \sum_{\tilde k \in \tilde Q(k)} E_{T_k,X^{k}_{T_{k}}} \big(  \frac{||\W^{\tilde k}||^2_2} {\rho(\tau_{\tilde k})^2} ||A||^2_2 E_{\tilde k} \big) + \nonumber \\
& 4 \frac{1}{N_i^2} \sum_{\tilde k \in \tilde Q(k)} \E_{T_k,X^k_{T_k} }\big(  1_{T_{\tilde k < T}} \frac{||\W^{\tilde k}||^2_2}{\rho(\tau_{\tilde k})^2}  ||A||^2_2 \quad || \widehat{D^2u}^{\tilde k} ||^2_2\big) +  \nonumber \\
&2 \frac{1}{N_i^2} \sum_{\tilde k \in \tilde Q(k)} \E_{T_k,X^k_{T_k} }\big( 1_{T_{\tilde k\ge T}} ||\W^{\tilde k}||^2_2 \frac{(\hat{g}^{\tilde k})^2}{ \Fb(T-T_k)^2} \big).
\end{align*}
We can iterate to get $E_\emptyset$ using the tower property

\begin{align}
\label{eq:E0a}
E_\emptyset \le & \prod_{i=1}^p \frac{1}{N_{i-1}} (1+\frac{8}{N_{i-1}}) \sum_{\tilde k^1 \in \tilde Q(\emptyset)} ... \sum_{\tilde k^p \in \tilde Q(\tilde{k}^{p-1})}  \E[ \prod_{i=2}^p \frac{ ||\W^{\tilde k^j}||_2}{\rho(\tau_{\tilde{k}^j})^2} \frac{||A||_2^{2p}}{\rho(\tau_{\tilde{k}^1})^2} E_{\tilde{k}^p}] + \nonumber \\
&  \sum_{i=0}^{p-1} \frac{1}{N_{i}^2} \prod_{j=1}^i \frac{1}{N_{j-1}} (1+\frac{8}{N_{j-1}})
\sum_{\tilde k^1 \in \tilde Q(\emptyset)} ... \sum_{\tilde k^{i+1} \in \tilde Q(\tilde{k}^{i})}  \E \big [ 1_{T_{\tilde{k}^{i+1} < T}}  \frac{||A||_2^{2i+2}}{\rho(\tau_{\tilde k^1})^2} \prod_{j=2}^{i+1} \frac{||\W^{\tilde k^j}||_2^2}{\rho(\tau_{\tilde k^j})^2}
4 || (\widehat{D^2u})^{\tilde k^{i+1}}||^2_2  +\nonumber \\
& 1_{T_{\tilde{k}^{i+1} > T}} 1_{T_{\tilde{k}^{i} < T}} \frac{2 ||A||_2^{2i} \displaystyle{\prod_{j=2}^{i+1}} ||\W^{\tilde k^j}||_2^2}{\displaystyle{\prod_{j=1}^{i}} \rho(\tau_{\tilde k^j})^2 \bar F(T-T_{\tilde{k}^{i}})^2}( (\widehat{g})^{\tilde k ^{i+1}})^2 \big ],
\end{align}
where $ E_{\tilde{k}^p} = 1_{T_{\tilde k^p < T}}  (\widehat{D^2g}(X_{T_{\tilde{k}^p}}^{\tilde{k}^p})- \widehat{D^2u}(T_{\tilde{k}^p},X_{T_{\tilde{k}^p}}^{\tilde{k}^p}))^2$.\\
Using Lemma \ref{lemma1} for function $D^2u-D^2g$ , and the fact that $\rho$ is bounded by below on $[0,T]$ by $\frac{1}{\hat C(T)} >0$:
we get that
\begin{align}
\E[ \prod_{i=2}^p \frac{ ||\W^{\tilde k^j}||_2}{\rho(\tau_{\tilde{k}^j})^2} \frac{1}{\rho(\tau_{\tilde{k}^1})^2} E_{\tilde{k}^p}] & \le  \sup_{t \in [0,T]} |D^{2p}u(t,.)-D^{2p}g|_\infty^2  \hat C(T)^{2p} C(\sigma)^{p-1}  \E( 1_{T_{\tilde k^p}} < T), \nonumber \\
& \le \tilde K^2 T^{2\theta} \hat C(T)^{2p} C(\sigma)^{p-1} \frac{\gamma(\alpha, \lambda T p)}{\Gamma(\alpha)}.
\label{eq1}
\end{align}
where we have used that $T_{\tilde k^p}$ follows a gamma law with parameters $(\alpha,p \lambda)$ and assumption \ref{ass::uProp}.\\
Similarly
\begin{align}
\E \big [ 1_{T_{\tilde{k}^{i+1} < T}}  \frac{1}{\rho(\tau_{\tilde k^1})^2} \prod_{j=2}^{i+1} \frac{||\W^{\tilde k^j}||_2^2}{\rho(\tau_{\tilde k^j})^2} || (\widehat{D^2u})^{\tilde k^{i+1}}||^2_2 ] \le \sup_{t \in [0,T]} |D^{2i+2}u(t,.)|_\infty^2 \hat C(T)^{2i+2}  C(\sigma)^{i} \frac{\gamma(\alpha, \lambda T (i+1))}{\Gamma(\alpha)},
\label{eq2}
\end{align}
and
\begin{align}
\E \big[ 1_{T_{\tilde{k}^{i+1}}\ge T} 1_{T_{\tilde{k}^{i}} < T}
\frac{\displaystyle{\prod_{j=2}^{i+1}} ||\W^{\tilde k^j}||_2^2}{ \displaystyle{\prod_{j=1}^{i}} \rho(\tau_{\tilde k^j})^2 \bar F(T-T_{\tilde{k}^{i}})^2}( (\widehat{g})^{\tilde k ^{i+1}})^2 \big ] &\le E \big[ 1_{T_{\tilde{k}^{i}} < T}\big] |D^{2i}g|_\infty^2 \frac{\hat C(T)^{2i} C(\sigma)^i}{\bar F(T)^2}, \nonumber\\
& \le |D^{2i}g|_\infty^2  \frac{\hat C(T)^{2i} C(\sigma)^i}{\bar F(T)^2} 
\frac{\gamma(\alpha, \lambda T i)}{\Gamma(\alpha)}.
\label{eq3}
\end{align}
Plugging equation \eqref{eq1}, \eqref{eq2}, \eqref{eq3} in \eqref{eq:E0a} gives the result.
\qed
\end{proof}
\begin{remark}
\label{remA}
 The case where $A$ depends on $t$ and $x$ is treated similarly. Instead of some bounds involving $\sup_{t\in [0,T]} |D^{2i}u(t,.)|_\infty$, we get some bounds
 involving $ sup_{t\in [0,T]} | \left( D^{2} A(t,.): \right)^{i-1} D^2 u(t,.)|_\infty$ such that
 it requires that $A(t,.)$ should have elements in $C^{2p}(\R^d)$.
 \end{remark}
This result gives us an algorithm to solve degenerated  Semi-Linear PDEs that cannot be solved with the algorithm given in \cite{warin2018nesting}.
Suppose that we want to solve:
\begin{flalign}
 \label{eqPDESemi}
  (-\partial_tu-\Lc u)(t,x)  & = f(t,x,u(t,x),Du(t,x)), \nonumber \\
  u(T,x)&=g(x), \quad  t<T,~x\in\R^d,
 \end{flalign}
 where now $\sigma$ is not invertible.
Then we introduce the operator
\begin{flalign}
 \hat \Lc u(t,x) := \mu Du(t,x) +  \frac{1}{2} \hat \sigma \hat \sigma^{\top} \!:\! D^2 u(t,x)
 \label{eq:genNonDeg}
 \end{flalign}
 such that $\hat \sigma$ is invertible. Then we can rewrite equation
 \eqref{eqPDESemi}  as:
 \begin{flalign}
 \label{eqPDESemiReform}
  (-\partial_tu-  \hat \Lc u)(t,x)  & = \tilde f(t, x,u(t,x),Du(t,x),D^2u(t,x))  \nonumber \\
  \tilde f(t, x,u(t,x),Du(t,x),D^2u(t,x)) := &f(t,x,u(t,x),Du(t,x)) - \frac{1}{2} (\hat \sigma \hat \sigma^{\top}- \sigma \sigma^{\top}) \!:\! D^2 u(t,x), \nonumber \\
  u(T,x)&=g(x), \quad  t<T,~x\in\R^d.
 \end{flalign}
 In order to have the converging result we have to take some assumptions  from \cite{warin2018nesting}:
 \begin{assump}
 \label{ass::lipFSemi}
 $f$ is  uniformly Lipschitz in $Du$ and $u$ with constant $K$ :
 \begin{flalign}
 | f(t, x, y,z) - f(t, x, \tilde y, \tilde z) | \le & K (|y-\tilde y| +  ||z-\tilde z||_2) \nonumber \\ &  \quad  \forall t \in [0,T], x \in \R^d, (y,\tilde y) \in \R \times \R, (z, \tilde z) \in \R^d \times \R^d.
 \end{flalign}
 \end{assump}
 \begin{assump}
 \label{ass::RegSemi}
  Equation  \eqref{eqPDESemi} has a solution $u \in C^{1,2p}([0,T] \x \R^d)$ with uniformly bounded derivatives in $x$ and $t$ and such that
  $D^{2p}u$ is $\theta$-H\"older with $\theta \in (0,1]$ in time following \eqref{eq:hol}
 \end{assump}
 Then using results in \cite{warin2018nesting} and proposition \ref{theo1}, we get the following proposition:
 \begin{proposition}
 Suppose that assumptions \ref{ass::lipFSemi} and \ref{ass::RegSemi} hold, then we have the following  error due to estimate \eqref{eq:estimFull1}  applied to equation \eqref{eqPDESemiReform} using a gamma Law with $0<\alpha<1$ for $\rho$ given by equation \eqref{rho}:
 \begin{align}
\label{est::finalSemi1}
 \E\big((\bar u_{\emptyset}^p - u(0,x))^2 \big) \le
 C_0(T,K,p)  + \sum_{i=1}^{p}  \frac{C_i(T,K)}{N_{i-1}}
\end{align}
where $C_0(T,K,p)$ goes to $0$ as $p$ goes to infinity, and $C_i$, $i>0$ are some functions depending on the maturity and the Lipschitz constant $K$ and going to $0$ as $i$ goes to infinity.
 \end{proposition}
 \begin{remark}
 The fact that the $c_i$ goes to zeros can be seen using Stirling formula  as in \cite{warin2018nesting}.
 \end{remark}

\section{Numerical results for the semi linear equations in the degenerated case.}
In this section we give an example of semi-linear equations where the diffusion coefficient of the SDE is not strictly bounded by below by a  strictly positive value. \\
The problem to solve is 
 \begin{flalign}
 \label{eqPDECIR}
  (-\partial_tu-\Lc u)(t,x)  & = f(x,u(t,x), Du(t,x))), \nonumber \\
  u_T&=g,
 \end{flalign}
 where
 \begin{flalign*}
 \Lc u(t,x) := k(m- x) Du(t,x) +  \frac{1}{2} \bar\sigma(x)^2 \!:\! D^2 u(t,x),
 \end{flalign*}
 and $k = \hat k I_d$, $\hat k \in \R^{+}$, $m  = \hat m \1_d$, $\hat m  \in \R^{+}$, $\bar \sigma(x)$ is a diagonal matrix with $\bar \sigma_{i,i}(x) =   \hat \sigma  \sqrt{x_i}$, $\hat \sigma \in \R^{+}$.\\
 Then the SDE associated corresponds to a multidimensional CIR process  where all component have the same dynamic :
 \begin{align}
 \label{eq:cir}
 dS_t^i  = \hat k ( \hat m - S^i) dt + \hat \sigma   \sqrt{S_t^i} dW_t^i
 \end{align}
 and $W_t^i$ are independent Brownian motions and such that the Feller
 condition $2 \hat k \hat m > \hat{\sigma}^2$ is satisfied.\\
The CIR simulation is generally tricky and necessitates the derivation of special schemes (see for example \cite{kahl2006fast}).
In order to avoid this simulation and the degeneracy of the diffusion coefficients, we rewrite equation \eqref{eqPDECIR}
as
\begin{align}
\label{eqPDECIR2}
  (-\partial_tu- \tilde \Lc u)(t,x)  = &  \tilde f(x, u(t,x), Du(t,x), D^2u(t,x)) , \nonumber \\
  \tilde f(x, u(t,x), Du(t,x), D^2u(t,x)) = & \frac{1}{2} (\bar\sigma(x)^2 - \tilde \sigma^2) D^2u(t,x)+ f(x,u(t,x), Du(t,x))), \nonumber \\
  \tilde \Lc u(t,x) := &k(m-  x) Du(t,x) +  \frac{1}{2} \tilde \sigma^2 \!:\! D^2 u(t,x), \nonumber \\
  \tilde \sigma = & \bar \sigma I_d,  \quad \bar \sigma \in \R^{+}
\end{align}
so that the associated SDE corresponds to a multidimensional Ornstein Uhlenbeck process where all components satisfy the same equation
\begin{align}
\label{eq:OU}
dS_t^i  = \hat k ( \hat m - S^i) dt + \bar \sigma   dW_t^i.
\end{align}
We apply our scheme to equation \eqref{eqPDECIR2}  using estimator \eqref{eq:estimFull1}. Note that theoretically, the regularity of $ A= \frac{1}{2} (\bar\sigma(x)^2 - \tilde \sigma^2)$ is not sufficient enough according to remark \ref{remA} but we will see that numerically the algorithm gives good results.
\\
A small adaptation of the scheme has to be achieved to deal with the fact that the coefficients are not constant. \\
In fact the SDE \eqref{eq:OU} can be solved exactly between two dates $t$ and $t + \Delta t$ introducing $\hat S_t \in \R^d$ with $(\hat S_t)_i = S^i_t$ using :
\begin{align}
\hat S_{t+\Delta t} =  A \hat S_t  + B  + C  G,
\end{align}
where $G$ is a vector composed of independent unit centered Gaussian variables, $ A = e^{ -\hat k \Delta t}  I_d$, $B = \hat m (1-e^{ -\hat k \Delta t} ) \1_d$, $C = \bar \sigma \sqrt{\frac{1- e^{ -2 \hat k \Delta t}}{2 \hat k \Delta t}} I_d$.
Therefore, the estimator \eqref{eq:estimFull1} has to be adapted replacing in the Malliavin weight $\sigma \sqrt{\Delta t}$ by $ A^{-1} \sigma$.\\
In our examples, we take the final function: 
\begin{align*}
g(x) = \cos(\sum_{i=1}^d x_i),
\end{align*}
the driver is taken as:
\begin{align*}
f(x,y, z)=  & a  y \sum_{i=1}^d z_i + (-\alpha + \sum_{i=1}^d \frac{\hat \sigma^2}{2} x_i )\cos(\sum_{i=1}^d x_i) e^{-\alpha (T-t)} + \\
& \sum_{i=1}^d \hat k ( \hat m -x_i) \sin(\sum_{i=1}^d x_i) e^{-\alpha (T-t)} +  a d \cos(\sum_{i=1}^d x_i) \sin(\sum_{i=1}^d x_i) e^{-2\alpha (T-t)} 
\end{align*}
such that there exists a regular solution given by 
\begin{align*}
u(t,x) = \cos(\sum_{i=1}^d x_i) e^{-\alpha (T-t)}.
\end{align*}
In all the examples, we take $a=0.1$, $\alpha =0.2$, $T=1$, $\hat k= 0.1$, $\hat m = 0.3$, $\hat \sigma =0.5$.
We have to choose a value for $\bar \sigma$. It is more effective to try to diminish the importance of the linear term so we take $\bar \sigma = \hat \sigma \sqrt{\hat m }$.\\
In the whole section the number of particles taken at each level will be a sequence $(N_i^{ipart})_{i\ge 0}$ indexed by $ipart$ such that:
\begin{align}
\label{repPart}
N_i^{ipart}= N_i^0 \times 2^{ipart}.
\end{align}
We take $\rho$ as the density of an exponential law so that $\rho(x)=e^{\lambda x}$. Theoretically
we have to take a Gamma law with $\alpha<1$ to treat the non linearity in $f$, but the use of $\alpha=1$ corresponding to the exponential case is numerically the most effective as shown in \cite{warin2018nesting}.
\begin{figure}[h!]
 \begin{minipage}[b]{0.49\linewidth}
  \centering
 \includegraphics[width=\textwidth]{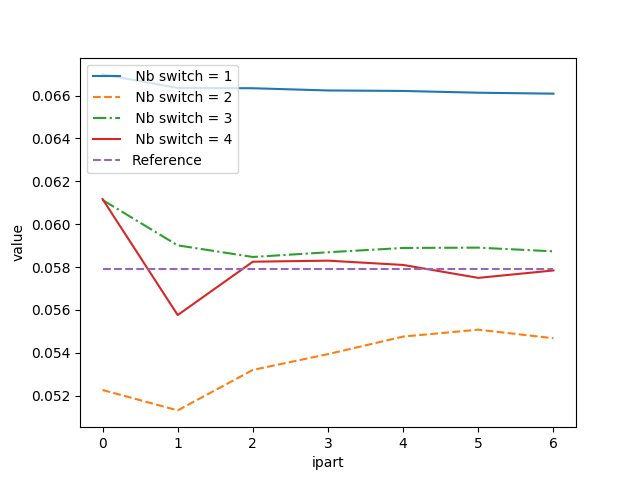}
 \caption*{$\lambda=0.1$.}
 \end{minipage}
\begin{minipage}[b]{0.49\linewidth}
  \centering
 \includegraphics[width=\textwidth]{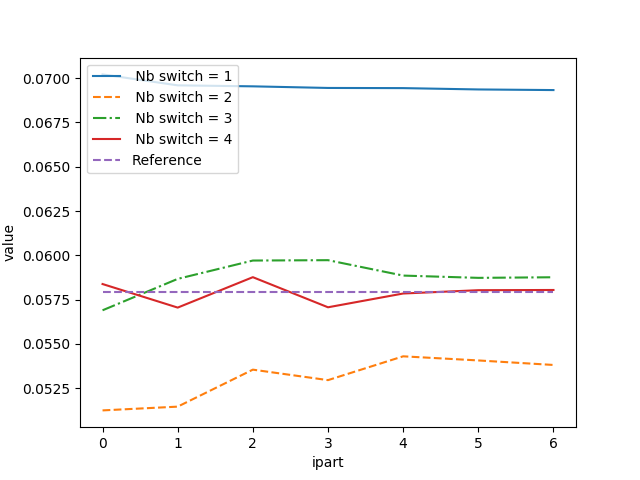}
 \caption*{$\lambda=0.15$.}
 \end{minipage}
 \caption{\label{fig:case1CIR5} CIR case dimension 5, $(N_0^0, N_1^0, N_2^0, N_3^0) = (1000,50,25,12)$}
\end{figure}
\begin{figure}[h!]
 \begin{minipage}[b]{0.49\linewidth}
  \centering
 \includegraphics[width=\textwidth]{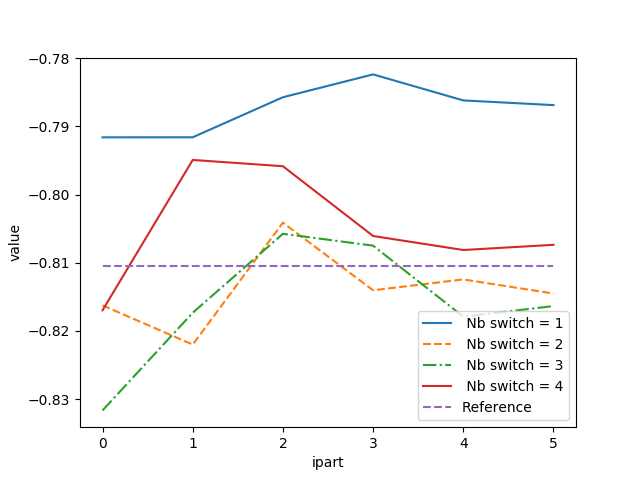}
 \caption*{$\lambda=0.1$.}
 \end{minipage}
\begin{minipage}[b]{0.49\linewidth}
  \centering
 \includegraphics[width=\textwidth]{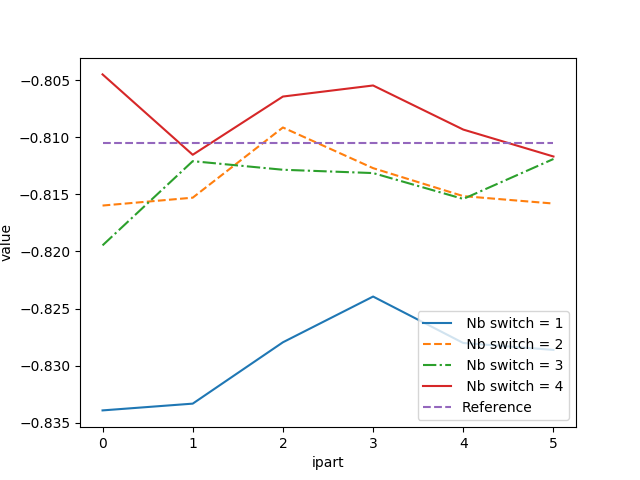}
 \caption*{$\lambda=0.15$.}
 \end{minipage}
 \caption{\label{fig:case1CIR10} CIR case dimension 10, $(N_0^0, N_1^0, N_2^0, N_3^0) = (1000,50,25,12)$}
\end{figure}
\begin{figure}[h!]
 \begin{minipage}[b]{0.49\linewidth}
  \centering
 \includegraphics[width=\textwidth]{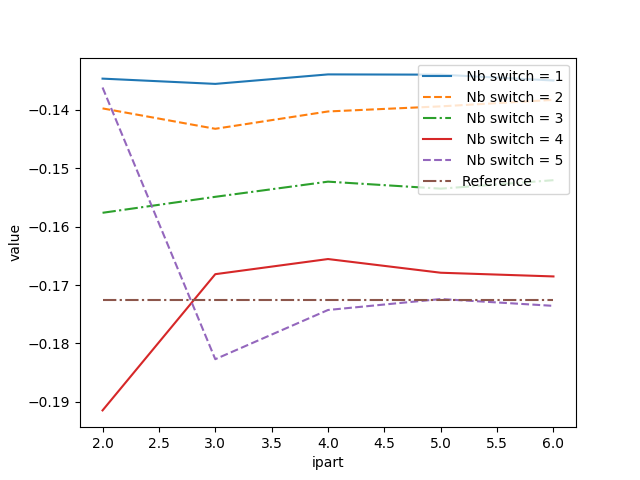}
 \caption*{$\lambda=0.05$.}
 \end{minipage}
\begin{minipage}[b]{0.49\linewidth}
  \centering
 \includegraphics[width=\textwidth]{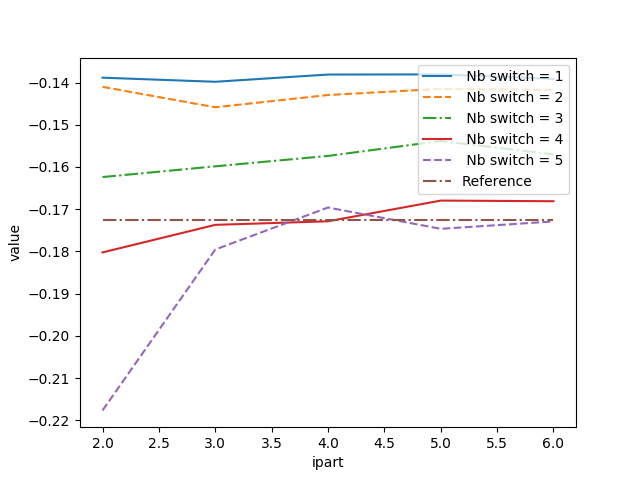}
 \caption*{$\lambda=0.075$.}
 \end{minipage}
 \caption{\label{fig:case1CIR15} CIR case dimension 15, $(N_0^0, N_1^0, N_2^0, N_3^0, N_4^0) = (1000,40,20,10,5)$}
\end{figure}
Results obtained are good but we have to take 4 switches to have a very good accuracy in dimension 5 and 10: we plot the results on figures \ref{fig:case1CIR5} and  \ref{fig:case1CIR10} taking  $\lambda = 0.1$ and $\lambda =0.15$.
In dimension $5$ a very accurate solution (with an error below $0.3\%$) is obtained taking at least $ipart$ equal to $4$ giving a computing time equal to $414$ seconds using $\lambda=0.15$ and $160$ seconds with $\lambda=0.1$. \\
In dimension $10$ the convergence is harder to reach and even if the results are good, the error seems to oscillate lightly. \\
On figure \ref{fig:case1CIR15} we give the results obtained in dimension 15: increasing the dimension, a number of $5$ switches is necessary and we take $\lambda=0.05$ and $\lambda=0.75$ to lower the computational time.
For example, for $5$ switches, $\lambda=0.075$ the error obtained is below $1\%$ for $ipart =5$ and $6$ for a computational time of $2800$ and $70000$ seconds. 
\newpage
\section{Numerical results for full non-linear equations}
As previously written, it was only proved that 
a driver linear in $D^2u$ was giving a converging method.
In this section we show numerically that the previous scheme can be used to solve some general HJB equations.
First we solve a toy problem  with a non linearity in $u D^2u$ in dimension 5 to 8.
At last we solve some problems of  continuous portfolio optimization.
\subsection{A first toy problem}
In this section we take the following parameters:
\begin{align*}
\mu= & \frac{\mu_0}{d} \un_d,\\
\sigma = & \frac{\sigma_0}{\sqrt{d}} \I_d,\\
f(t,x,y,z,\theta)=& \cos(\sum_{i=1}^d x_i) (\alpha +\frac{1}{2}\sigma_0^2)
e^{\alpha (T-t)}+ \sin(\sum_{i=1}^d x_i) \mu_0 e^{\alpha (T-t)} +
a \sqrt{d} \cos(\sum_{i=1}^d x_i)^2 e^{2\alpha (T-t)} \\
& + \frac{a}{\sqrt{d}} ( - e^{2\alpha (T-t)} ) \vee ( e^{2\alpha (T-t)} \wedge ( y \sum_{i=1}^d \theta_{i,i})),
\end{align*}
with $g(x)= \cos(\sum_{i=1}^d x_i)$, such that an explicit solution is given by
\begin{align*}
u(t,x)= e^{\alpha (T-t)} \cos( \sum_{i=1}^d x_i).
\end{align*}
We set $\mu_0=0.2$, $\sigma_0=1$, $\alpha=0.1$, $x_0=0.5 \un_d$, $T=1$.
\begin{figure}[h!]
 \begin{minipage}[b]{0.49\linewidth}
  \centering
 \includegraphics[width=\textwidth]{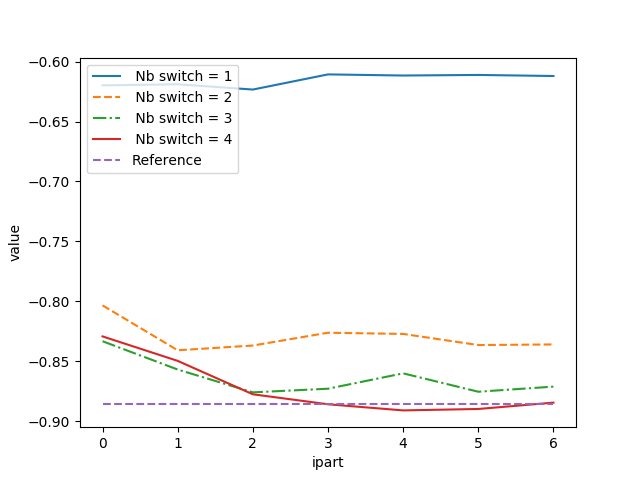}
 \caption*{$\lambda=0.1$.}
 \end{minipage}
\begin{minipage}[b]{0.49\linewidth}
  \centering
 \includegraphics[width=\textwidth]{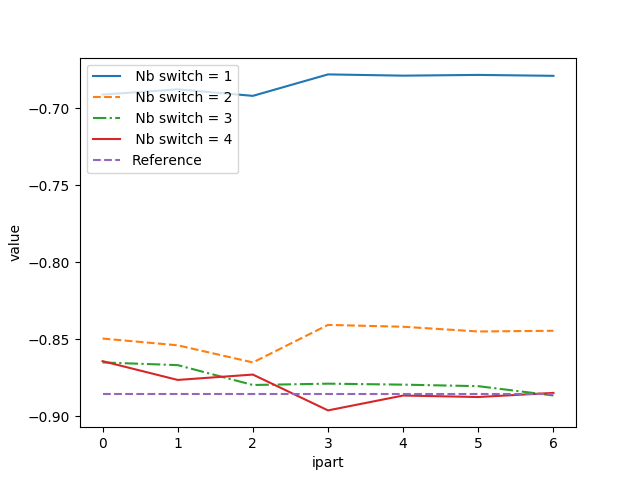}
 \caption*{$\lambda=0.2$.}
 \end{minipage}
 \caption{\label{fig:caseFullNL1} Full non linear toy example $a=0.1$, $d=5$.}
\end{figure}
\begin{figure}[h!]
 \begin{minipage}[b]{0.49\linewidth}
  \centering
 \includegraphics[width=\textwidth]{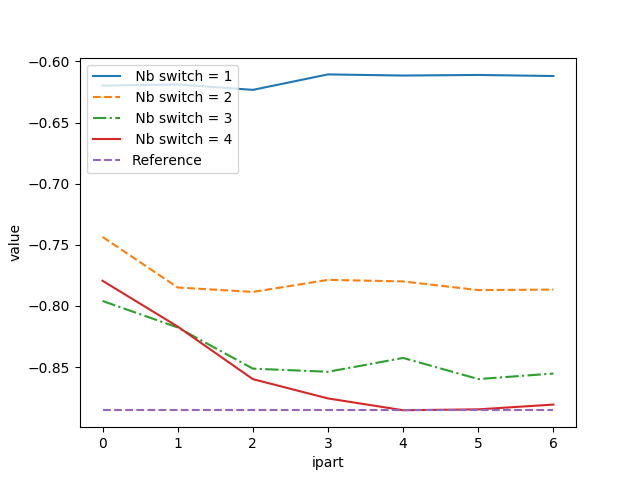}
 \caption*{$\lambda=0.1$.}
 \end{minipage}
\begin{minipage}[b]{0.49\linewidth}
  \centering
 \includegraphics[width=\textwidth]{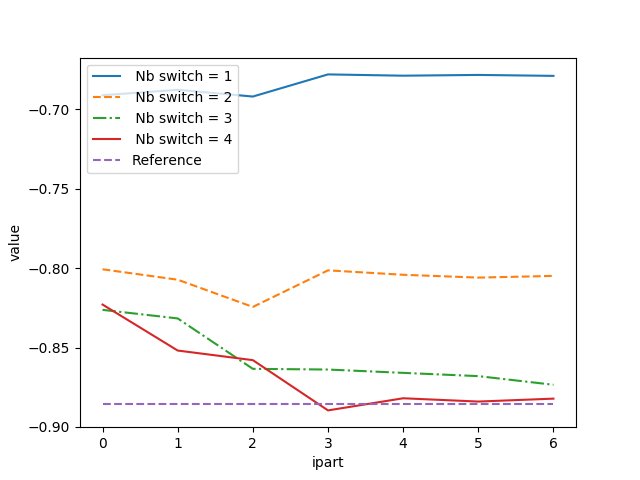}
 \caption*{$\lambda=0.2$.}
 \end{minipage}
 \caption{\label{fig:caseFullNL2} Full non linear toy example $a=0.2$, $d=5$.}
\end{figure}
\begin{figure}[h!]
 \begin{minipage}[b]{0.49\linewidth}
  \centering
 \includegraphics[width=\textwidth]{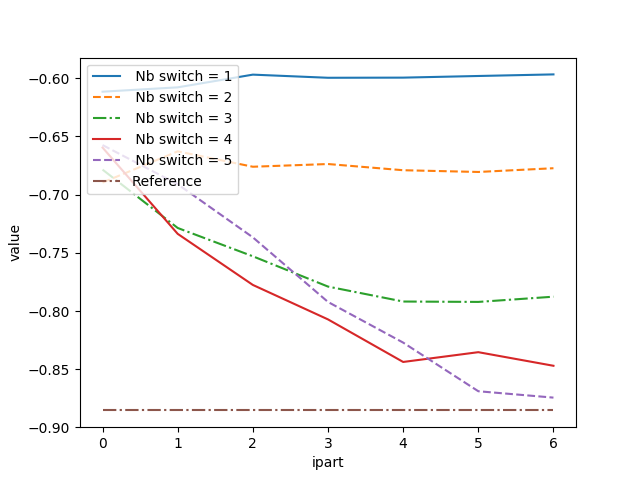}
 \caption*{$\lambda=0.1$.}
 \end{minipage}
\begin{minipage}[b]{0.49\linewidth}
  \centering
 \includegraphics[width=\textwidth]{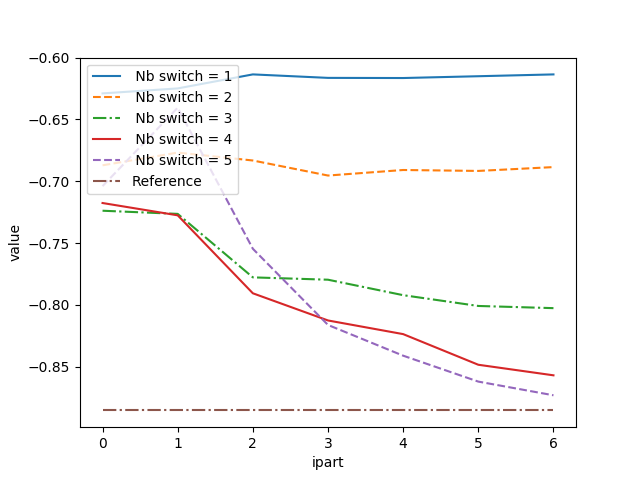}
 \caption*{$\lambda=0.1$.}
 \end{minipage}
 \caption{\label{fig:caseFullNL3} Full non linear toy example $a=0.4$, $d=5$.}
\end{figure}
All results are obtained using a number of particles given by \eqref{repPart}
with $(N_0,N_1,N_2,N_3,N_4)= (1000,40,40,20,20)$ .
On figures \ref{fig:caseFullNL1}, \ref{fig:caseFullNL2}, \ref{fig:caseFullNL3},  we plot  the result obtained  in dimension 5 for different values of $a$.
Clearly for $a=0.1$, $a=0.2$, the solution is reached with $4$ switches, while
it is not the case for $a=0.4$: 5 switches are necessary to get an accurate solution and on the graph the slope of the curve for a number of switches equal to $5$ for $ipart$ between $4$ and $6$ clearly indicates that a value $ipart=7$ should increase the accuracy.\\
On figure \ref{fig:caseFullTime}, we see the time explosion in dimension $5$ for $4$ switches as a function of $ipart$.
\begin{figure}[h!]
  \centering
 \includegraphics[width=0.5\textwidth]{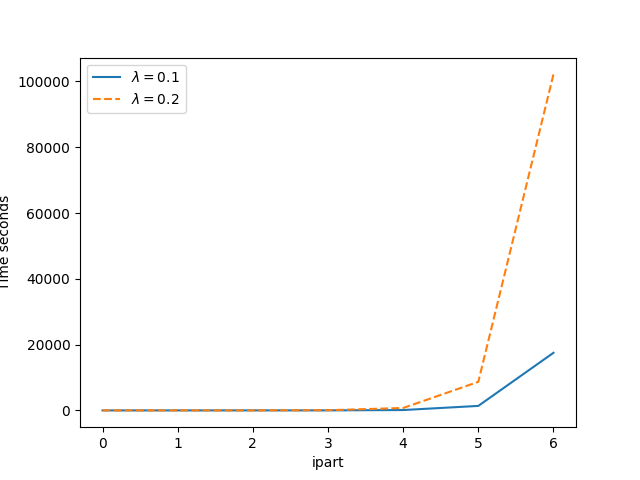}
 \caption{\label{fig:caseFullTime} Computational time  for $d=5$, $4$ switches.}
\end{figure}
At last on figure \ref{fig:caseFullNL4}, we plot the solution obtained in dimension $d=7$. The results are always very good but of course the error is higher than in dimension $5$.
\begin{figure}[h!]
 \begin{minipage}[b]{0.49\linewidth}
  \centering
 \includegraphics[width=\textwidth]{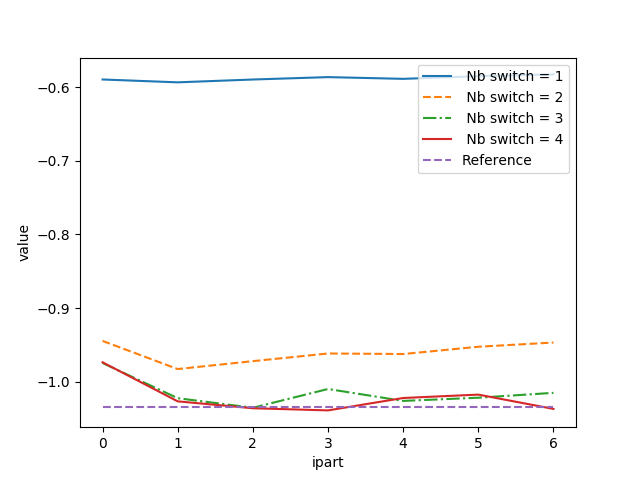}
 \caption*{$\lambda=0.1$.}
 \end{minipage}
\begin{minipage}[b]{0.49\linewidth}
  \centering
 \includegraphics[width=\textwidth]{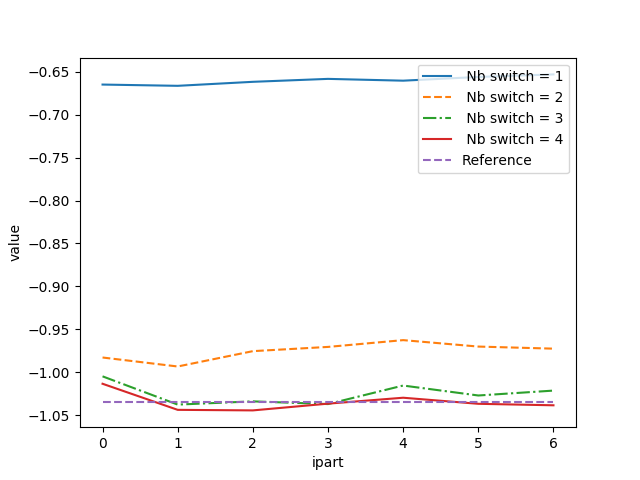}
 \caption*{$\lambda=0.2$.}
 \end{minipage}
 \caption{\label{fig:caseFullNL4} Full non linear toy example $a=0.1$, $d=7$.}
\end{figure}
\newpage
\subsection{Some HJB problems}
We solve the  problem of  continuous portfolio optimization in dimension two in a special case where we have semi-analytical solutions.
In this whole section we consider an investor who has access to  some non risky asset $S^0$  and $n$ risky assets.
The non-risky asset $S^0$  has a $0$ return so $dS^0_t=0$, $t\in[0,1]$. The dynamic of the $n$ risk assets is given by $\{S_t,t\in[0,T]\}$  an It\^o process.
The investor chooses an adapted process $\{\kappa_t,t\in[0,T]\}$ with values in $\R^n$, where $\kappa^i_t$ is the amount he decides to invest into asset $i$.\\
The portfolio dynamic is given by:
 \begin{align*}
 dX^\kappa_t
 =
 \kappa_t\cdot \frac{dS_t}{S_t} +(X^\kappa_t-\kappa_t\cdot \1)\frac{d S^0_t}{S^0_t}
 = 
 \kappa_t\cdot \frac{dS_t}{S_t}.
 \end{align*}
Let $\Ac$ be the collection of all adapted processes $\kappa$ with values in $\R^d$ and which are integrable with respect to $S$. Given an absolute risk aversion coefficient $\eta>0$, the portfolio optimization problem is defined by:
 \begin{align}
 \label{prob-portefeuille}
 v_0
 :=&
 \sup_{\kappa\in\Ac} \E\left[-\exp\left(-\eta X^\kappa_T\right)\right].
 \end{align}
\subsubsection{A first two dimensional problem}
We take this problem from \cite{fahim2011probabilistic}.
Let's take $n=1$ and assume that the security price process is defined by the Heston model \cite{heston1993closed}:
 \begin{align*}
 dS_t =& \mu S_tdt +  \sqrt{Y_t} S_t dW_t^{(1)}  
 \\
 dY_t =&  k (m-Y_t)dt + c\sqrt{Y_t} \left(\rho dW_t^{(1)}+\sqrt{1-\rho^2}dW_t^{(2)}\right),
 \end{align*}
where $W=(W^{(1)},W^{(2)})$ is a Brownian motion in $\R^2$.
As pointed out in \cite{fahim2011probabilistic}, the portfolio optimization problem \eqref{prob-portefeuille} does not depend on $S_t$. Given an initial state at the time origin $t$ given by $(X_t,Y_t)=(x,y)$, the value function $v(t,x,y)$ solves the HJB equation:
 \begin{equation}\label{hjb0}
 \begin{array}{rl}
 v(T,x,y) = - e^{-\eta x}
 ~\mbox{and}~
 0 
 =&
 - v_t -  k (m-y) v_y - \frac{1}{2}c^2 y v_{yy}
 -\sup_{\kappa\in\R} \left(\frac12\kappa^2yv_{xx}+\kappa(\mu v_x+\rho cy v_{xy})\right) 
 \\
 =&
 - v_t -  k (m-y) v_y - \frac{1}{2}c^2 y v_{yy}
 + \frac{(\mu v_x + \rho cy v_{xy})^2} 
        {2 y v_{xx}}.
 \end{array}
 \end{equation}
A quasi explicit solution of this problem was provided by Zariphopoulou \cite{zariphopoulou2001solution}:
 \begin{align}
 \label{zariphopoulou}
 v(t,x,y)=-e^{-\eta x} \left\| \exp\left(-\frac{1}{2}\int_t^T \frac{\mu^2}{\tilde Y_s}ds\right)
                       \right\|_{\L^{1-\rho^2}}
 \end{align}
 where the process $\tilde Y$ is defined by
 \begin{align*}
 \tilde Y_t=y \quad
 \mbox{  and} \quad &
 d\tilde Y_t = ( k (m-\tilde Y_t)- \mu c\rho)dt +  c\sqrt{\tilde Y_t} dW_t.
 \end{align*}
 Choosing $ \bar \sigma >0$, we can rewrite the problem as equation  \eqref{eqPDEFull} where
 \begin{align*}
 \mu = & ( 0, k (m-y))^{\top},  \qquad
 \sigma = \left(  \begin{array}{ll}
 \bar \sigma & 0 \\
 0 & c \sqrt{m}
 \end{array} \right), \qquad
 g(x) = - e^{-\eta x}
 \end{align*} 
 and 
 \begin{align}
 f(x,y,z,\theta)= &
 -\frac{1}{2} \bar{\sigma}^2 \theta_{11}  +\frac{1}{2} c^2 (y^2-m) \theta_{2,2} -\frac{(\mu z_1 + \rho c y \theta_{12})^2}{2 y \theta_{11}}.
 \end{align}
In order to have  $f$ Lipschitz, we truncate the control limiting the amount invested by taking
\begin{align*}
f_{M}(y,z,\theta) =
-\frac{1}{2} \bar \sigma^2 \theta_{11}  +\frac{1}{2} c^2 (y^2-m) \theta_{2,2} +
 \sup_{ 0 \le\eta\le M} \left(\frac{1}{2}\eta^2 y \theta_{11}+\eta (\mu z_1+\rho c y \theta_{12})\right).
\end{align*}
We take the following parameters :
$\mu =0.05$, $c  = 0.2$, $ k  = 0.1$, $m  = 0.3$, $Y_0 =  m$, $\rho =0$, $\eta=1$. The initial value of the portfolio is $x_0=1$, the maturity $T$ is taken equal to one year, giving a value function  $v_0=-0.3662$ computed from the quasi-explicit formula \eqref{zariphopoulou}.
On figure \ref{fig:case1Port1}, we give the results obtained by taking $\bar \sigma =0.1$ with one and two switches, which is enough to get a very accurate solution.
For $ipart=8$ and two switches we obtain $0.3661$ for both $\lambda =0.1$ and  $\lambda=0.15$.
\begin{figure}[h!]
 \begin{minipage}[b]{0.49\linewidth}
  \centering
 \includegraphics[width=\textwidth]{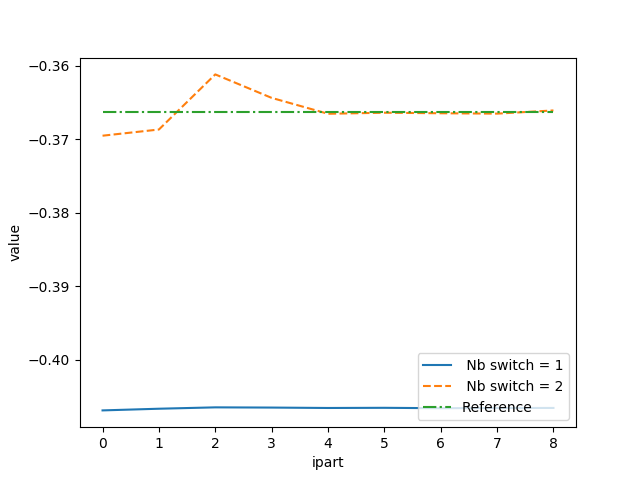}
 \caption*{$\lambda=0.1$.}
 \end{minipage}
\begin{minipage}[b]{0.49\linewidth}
  \centering
 \includegraphics[width=\textwidth]{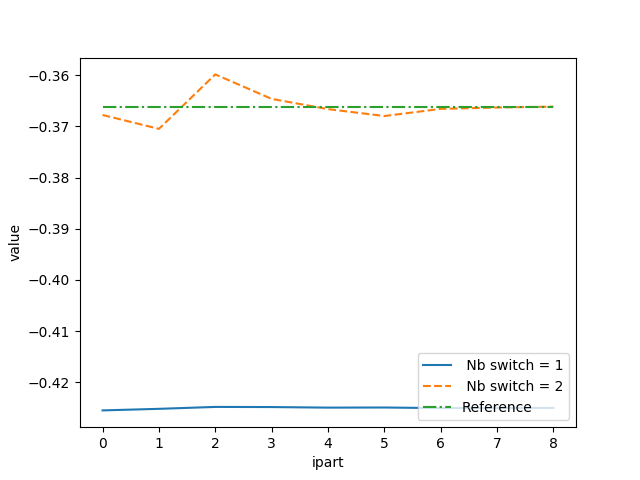}
 \caption*{$\lambda=0.15$.}
 \end{minipage}
 \caption{\label{fig:case1Port1} Portfolio optimization, $d=1$, $M=4$, $\bar \sigma =0.1$, $(N_0^0, N_1^0) = (1000,40)$}
\end{figure}
On figure \ref{fig:case1Port2}, we give the results obtained by taking $\bar \sigma =0.2$.
For $ipart=8$, we obtain $0.3654$ for $\lambda =0.1$ and $0.3658$ for $\lambda=0.15$ which is quite as not good as with $\bar \sigma =0.1$.
\begin{figure}[h!]
 \begin{minipage}[b]{0.49\linewidth}
  \centering
 \includegraphics[width=\textwidth]{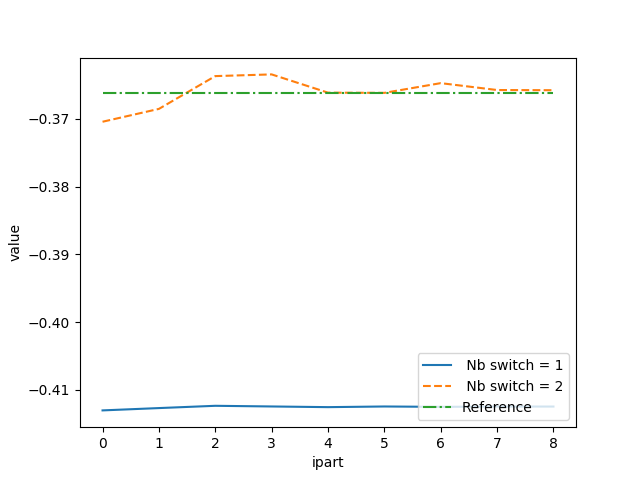}
 \caption*{$\lambda=0.1$.}
 \end{minipage}
\begin{minipage}[b]{0.49\linewidth}
  \centering
 \includegraphics[width=\textwidth]{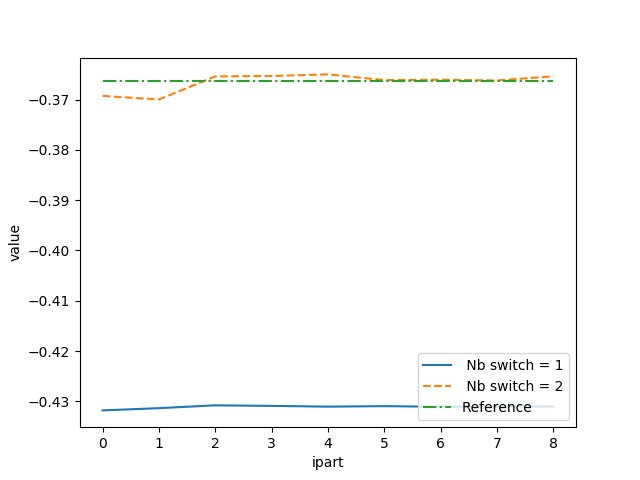}
 \caption*{$\lambda=0.15$.}
 \end{minipage}
 \caption{\label{fig:case1Port2} Portfolio optimization, $d=1$, $M=4$, $\bar \sigma =0.2$, $(N_0^0, N_1^0) = (1000,100)$}
\end{figure}
\subsubsection{In higher dimensions}
We assume that we dispose of $d$  securities all of them  being defined by a Heston model:
 \begin{align*} 
 dS_t^i =& \mu^i S_t^i dt +  \sqrt{Y^i_t} S^i_t dW_t^{(2i-1)}  \\
 dY_t^i =&  k^i (m^i-Y_t^i)dt + c^i \sqrt{Y^i_t} dW_t^{(2i)},
 \end{align*}
where $W=(W^{(1)}, ... , W^{(2d)})$ is a Brownian motion in $\R^{2d}$.
As in the two dimensional case, the problem doesn't depend on the $s^i$.
As in \cite{zariphopoulou2001solution}, we can guess that the solution can be expressed as
$$
v(t,x,y^1,..,y^d) = e^{-\eta x} u(t,y^1,..., y^d),
$$
and using Feyman Kac it is easy to see that then  a general solution
can be written
\begin{align}
 \label{zariphopoulouND}
 v(t,x,y^1,.., y^d)=-e^{-\eta x} \E[\prod_{i=1}^d \exp\left(-\frac{1}{2}\int_t^T \frac{(\mu^i)^2}{\tilde Y^i_s}ds \right) ]
 \end{align}
 with
 \begin{align*}
 \tilde Y_t^i=y^i \quad
 &\mbox{  and}&
 d\tilde Y_t^i = k^i (m^i-\tilde Y_t^i)dt +  c^i\sqrt{\tilde Y^i_t} dW^i_t,
 \end{align*}
 where $y^i$ corresponds to the initial value of the volatility at date $0$ for asset $i$.\\
 Choosing $ \bar \sigma >0$, we can write the problem as equation  \eqref{eqPDEFull} in dimension $d+1$  where
 \begin{align*}
 \mu = & ( 0, k^1 (m^1-y^1), ...,k^d (m^d-y^d) )^{\top},  \qquad
 \sigma = \left(  \begin{array}{lllll}
 \bar \sigma & 0 & ... & ...&  0 \\
 0 & c \sqrt{m^1} & 0 & ... & 0 \\
 0  & \dotsb &  \ddots & \dotsb  & 0 \\
 0  & \dotsb &  \dotsb & \ddots  & 0 \\
 0 & ... & ... & 0 & c \sqrt{m^d}
 \end{array} \right)
 \end{align*}
 always with the same terminal condition
 \begin{align*}
 g(x) = - e^{-\eta x}
 \end{align*} 
 and 
 \begin{align}
 f(x,y,z,\theta)= &
 -\frac{1}{2} \bar{\sigma}^2 \theta_{11}  +\frac{1}{2} \sum_{i=1}^d (c^i)^2 ((y^i)^2-m^i) \theta_{i+1,i+1} -  \sum_{i=1}^d  \frac{\mu^i z_1 }{2 y^i \theta_{11}}.
 \end{align}
Once again, in order to have  $f$ Lipschitz, we truncate the control limiting the amount invested by taking
\begin{align*}
f_{M}(y,z,\theta) = &
-\frac{1}{2} \bar \sigma^2 \theta_{11}  +\frac{1}{2} \sum_{i=1}^d (c^i)^2 ((y^i)^2-m^i) \theta_{2,2} + \\
&
 \sup_{ \begin{array}{c}
 \eta = (\eta^1,...,\eta^d) \\
 0 \le \eta^i\le M, i=1,d
 \end{array}}   \sum_{i=1}^d \left(\frac{1}{2}(\eta^i)^2 y^i \theta_{11}+(\eta^i) \mu^i z_1\right).
\end{align*}
 \\
 We suppose in our example that all assets have the same parameters that are equal to the parameters taken in the two dimensional case. We also suppose that the initial conditions are the same as before.\\
 Taking $\bar \sigma =0.2$, for $d=3$, $d=8$, $d=10$, we give the results obtained with one and two switches on figures \ref{fig:case1Portd3},\ref{fig:case1Portd8}, \ref{fig:case1Portd10}.
\begin{figure}[h!]
 \begin{minipage}[b]{0.49\linewidth}
  \centering
 \includegraphics[width=\textwidth]{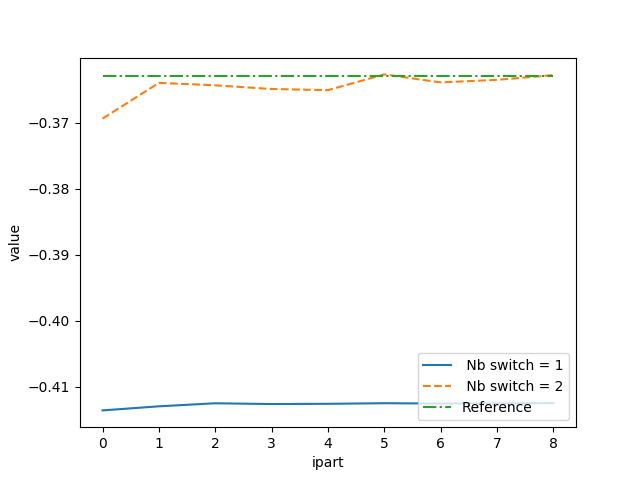}
 \caption*{$\lambda=0.1$.}
 \end{minipage}
\begin{minipage}[b]{0.49\linewidth}
  \centering
 \includegraphics[width=\textwidth]{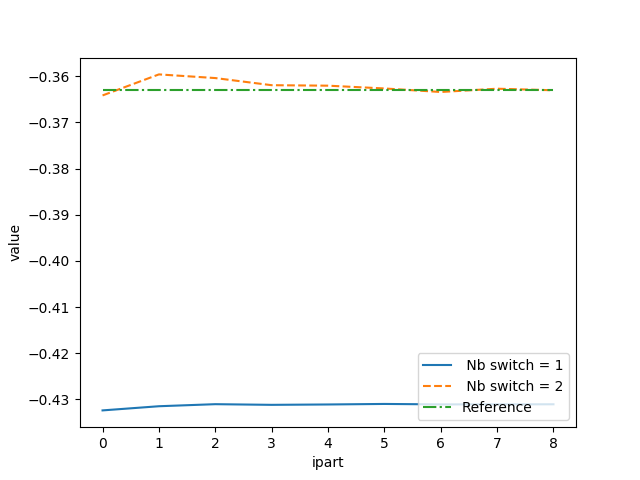}
 \caption*{$\lambda=0.15$.}
 \end{minipage}
 \caption{\label{fig:case1Portd3} Portfolio optimization, $d=3$, $M=4$, $\bar \sigma =0.2$, $(N_0^0, N_1^0) = (1000,100)$.}
\end{figure}
\begin{figure}[h!]
 \begin{minipage}[b]{0.49\linewidth}
  \centering
 \includegraphics[width=\textwidth]{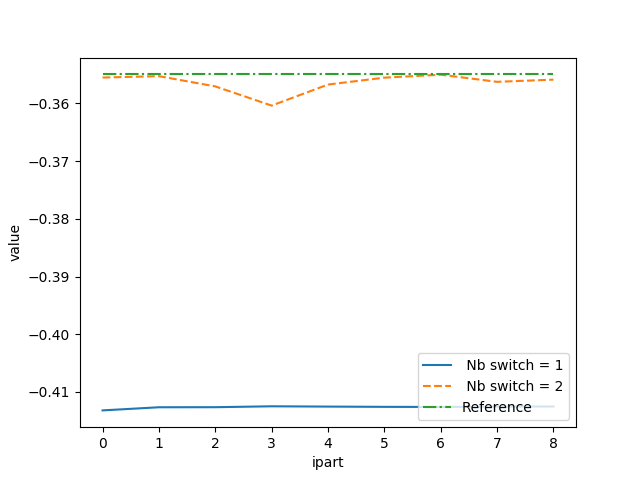}
 \caption*{$\lambda=0.1$.}
 \end{minipage}
\begin{minipage}[b]{0.49\linewidth}
  \centering
 \includegraphics[width=\textwidth]{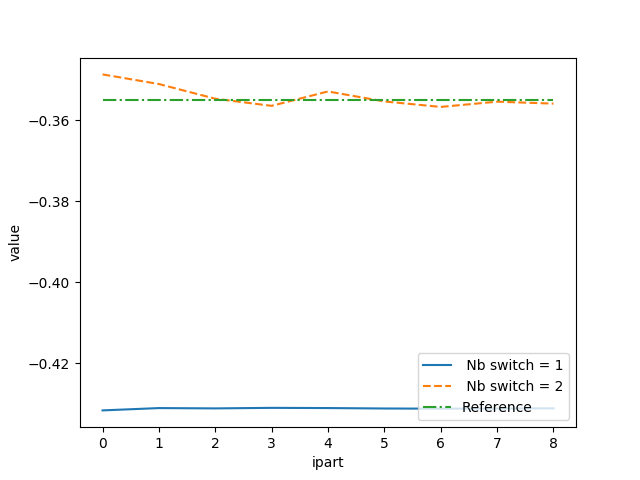}
 \caption*{$\lambda=0.15$.}
 \end{minipage}
 \caption{\label{fig:case1Portd8} Portfolio optimization, $d=8$, $M=4$, $\bar \sigma =0.2$, $(N_0^0, N_1^0) = (1000,100).$}
\end{figure}
\begin{figure}[h!]
 \begin{minipage}[b]{0.49\linewidth}
  \centering
 \includegraphics[width=\textwidth]{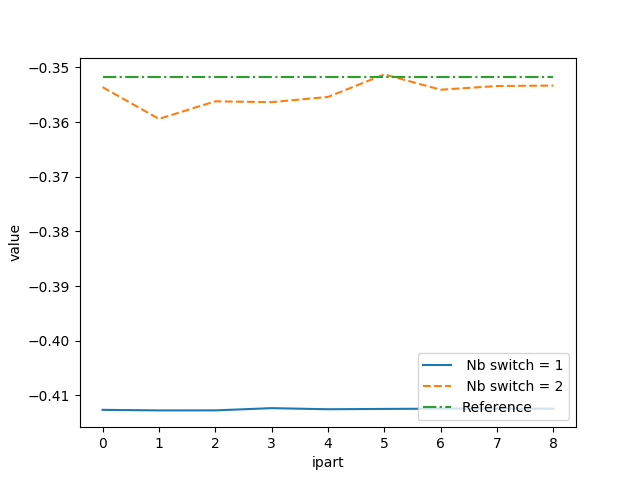}
 \caption*{$\lambda=0.1$.}
 \end{minipage}
\begin{minipage}[b]{0.49\linewidth}
  \centering
 \includegraphics[width=\textwidth]{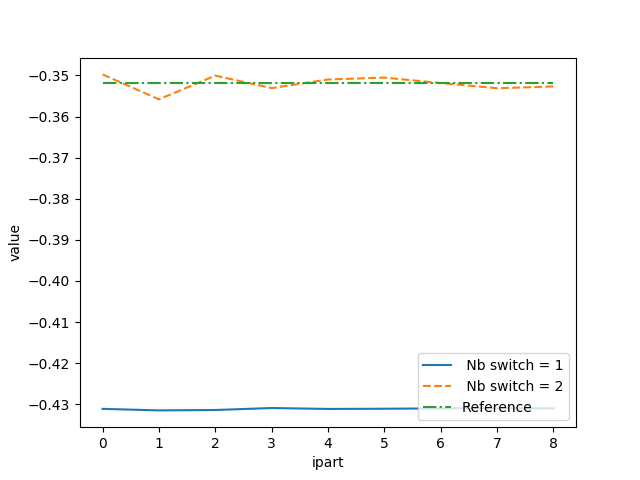}
 \caption*{$\lambda=0.15$.}
 \end{minipage}
 \caption{\label{fig:case1Portd10} Portfolio optimization, $d=10$, $M=4$, $\bar \sigma =0.2$, $(N_0^0, N_1^0) = (1000,100)$.}
\end{figure}
Results obtained are very accurate and the result are all obtained in less than $20$ seconds.

\newpage

\section{Conclusion}
An effective method to solve degenerated semi-linear equation in high dimension has been developed and is proved to be converging. Numerically it can be shown that it can be used to solve some full non linear problems.
The results are similar to the one in \cite{warin2018nesting}: the resolution time is linear with the dimension of the problem and to get accurate solutions in a reasonable computational time it is necessary to have  the Lipschitz constant of the problem and the maturity of the problem not too high.

\section{Ackowledgements}
This work has benefited from the financial support of the ANR Caesar and ANR program "Investissement d'avenir"

\bibliographystyle{spmpsci}
\bibliography{MyBib}

\end{document}